\documentclass[11pt,reqno]{amsart}
\usepackage[utf8]{inputenc}
\usepackage{tikz-cd}
\usepackage{graphicx}
\usepackage{amssymb,amsmath, amscd, latexsym,amsfonts,bbm, amsthm,stmaryrd, enumerate,phaistos}
\usepackage{yhmath}
\usepackage{comment}\usepackage{todonotes}\usepackage[all]{xy}
\usepackage[vcentermath]{youngtab}

\usepackage{amsmath, amssymb, latexsym, enumerate,  graphicx,tikz, stmaryrd,pifont,ifsym}
\usetikzlibrary{arrows,decorations,decorations.pathmorphing,positioning}
\usepackage{mdframed}
\usepackage{mdwlist}
\usepackage[pdftex]{hyperref}
\usepackage{amscd,mathrsfs,epic,empheq,float}
\usepackage{bbm}
\usepackage{cases}
\usepackage[all]{xy}

\usepackage{tikz}
\usepackage{tikz-cd}

\theoremstyle{definition}

\newcommand{\CC}{\mathbb C} 

\newcommand{\ZZ}{\mathbb Z}

\usepackage{pict2e}
\def\StrangeCross

\usepackage{amsfonts}
\usepackage{amssymb}
\usepackage{amscd}
\usepackage{amsmath}
\def\bea{\begin{eqnarray}}
\def\eea{\end{eqnarray}}
\def\nn{\nonumber}

\def\ZZ{\mathbb{Z}}
\def\CC{\mathbb{C}}

\def\hh{\mathfrak{h}} 
\def\nn{\mathfrak{n}}
\def\gg{\mathfrak{g}} 

\usepackage{pict2e}

\def\endproof{\hfill$\square$\medskip}

\newtheorem{theorem}{Theorem}[section]

\newtheorem{lemma}[theorem]{Lemma}
\newtheorem{proposition}[theorem]{Proposition}
\newtheorem{problem}[theorem]{Problem}
\newenvironment{exafont}{\begin{bf}}{\end{bf}}

\newtheorem{corollary}[theorem]{Corollary}
\theoremstyle{definition}
\newtheorem{definition}[theorem]{Definition}
\newtheorem{remark}[theorem]{Remark}
\newtheorem{example}[theorem]{Example}

\begin{document}


\title{Generalized electrical Lie algebras}

\author[Arkady Berenstein]{Arkady~Berenstein}
\address{A.~B.: University of Oregon,
Mathematics Departments, USA }
\email{arkadiy@uoregon.edu}
\author[Azat Gainutdinov]{Azat~Gainutdinov}
\address{A.~G.:  Laboratoire de Mathématiques et Physique Théorique CNRS/UMR 7350,
Fédération Denis Poisson FR2964, Université de Tours, Parc de Grammont, 37200 Tours, France}
\email{gainut@gmail.com}
\author[Vassily~Gorbounov]{Vassily~Gorbounov}
\address{V.~G.: Faculty of Mathematics, National Research University Higher School of Economics, Usacheva 6, 119048 Moscow, Russia}
\email{vgorb10@gmail.com}

\thanks{This work was partially supported by the Simons Foundation Collaboration Grant for Mathematicians no.~636972 (AB)
}

\begin{abstract}
We generalize the electrical Lie algebras originally introduced by Lam and Pylyavskyy in several ways. To each Kac-Moody Lie algebra $\gg$ we associate two types (vertex type and edge type) of the generalized electrical algebras.
The electrical Lie algebras of vertex type are always  subalgebras of $\gg$ and are flat deformations of the nilpotent Lie subalgebra of $\gg$. In many cases including $sl_n$, $so_n$, and $sp_{2n}$ we find new (edge) models for  our generalized electrical Lie algebras of vertex type. Finding an edge model in general is an interesting open problem. 
\end{abstract}

\maketitle
\setcounter{tocdepth}{3}
\tableofcontents
MSC2020: MSC2020: 17B05, 17B10, 05E10, 17B67, 82B20


\section{Introduction and main results}
The electrical Lie algebras were introduced by T. Lam and P. Pylyavskyy in \cite{LP} and further studied by Yi Su in \cite{su}.
An infinitely generated electrical Lie algebra of type $A$ also appeared in the context of categorification in representation theory \cite{ss16}.
The aim of this paper is to generalize the notion of electrical Lie algebra into a multiparametric family in any given
semisimple or Kac-Moody Lie algebra $\gg$, in such a way that they are flat deformations of the nilpotent part $\nn$ of $\gg$. We also construct various embeddings (vertex and edge models) of the electrical Lie algebras into the corresponding Kac-Moody ones.

We start with the definition of the electrical Lie algebra given below which depends on a set of parameters $a_i$. If all $a_i$ are equal to the imaginary unit these generators in the case of $sl_n$ give an important representation of the Temperley-Lieb algebra for the zero value of the loop fugacity parameter (\cite{Azat}) and they also give a solution to the Zamolodchikov tetrahedral equation \cite{KKS}, \cite{GT}.

\begin{definition}
Let $A=(a_{ij})$ be a (generalized, not necessarily symmetrizable) $I\times I$ Cartan matrix and let $\gg=\gg_A=<e_i,f_i,i\in I>$ be the corresponding Kac-Moody Lie algebra. For any family  ${\bf a}=(a_i,{i\in I})\in  \CC^I$ let
$\gg^{({\bf a})}$ be a  Lie subalgebra of $\gg$  generated by 
\begin{equation}
\label{eq:electric generators}
u_i:=e_i+a_i[e_i,f_i]-a_i^2f_i\ , i\in I.
\end{equation}
We call $\gg^{({\bf a})}$ a {\it generalized electrical Lie algebra of type} $\gg$.
\end{definition} 

In particular,  $sl_2^{({\bf a})}$ is generated by a single nilpotent $u=\begin{pmatrix} a & 1\\
-a^2 & -a
\end{pmatrix}\in sl_2$.

This definition makes sense  if the ground field $\CC$ is replaced with any algebra containing all $a_i$, e.g., with $\CC[a_i,i\in I]$ and it is justified by the following.  

\begin{theorem}
\label{th:canonical electric} For any Kac-Moody Lie algebra $\gg$ and any ${\bf a}\in \CC^I$ the generalized electrical Lie algebra
$\gg^{({\bf a})}$ of type $\gg$ satisfies 
\begin{equation}
\label{eq:deformed Serre}  
(ad~u_i)^{1-a_{ij}}(u_j)=-2\delta_{a_{ij},-1}a_{ji}a_ia_ju_i
\end{equation}
for all distinct $i,j\in I$.
\end{theorem}

It turns out that these relations are defining. This was conjectured for finite types in \cite[Conjecture 5.2]{LP}, which was partially confirmed in \cite{su} and \cite{CGJ,CY}. 

\begin{theorem}  
\label{th:flat electric}
In the notation of Theorem \ref{th:canonical electric} the relations \eqref{eq:deformed Serre} provide a presentation of $\gg^{({\bf a})}$. Moreover, 
$\gg^{({\bf a})}$ is a flat deformation of the nilpotent part ${\mathfrak n}=<e_i,i\in I>$ of $\gg$ in the sense that $\gg^{({\bf a})}\cong \nn$ for any ${\bf a}\in \CC^I$ as  naturally filtered vector spaces.
\end{theorem}

We prove Theorem \ref{th:flat electric} in Section \ref{subsec:proof of Theorem flat electric} (along with its generalization to any coefficient algebra $R$ containing all $a_i$) by utilizing a mini-theory of flat deformations of associative and Lie algebras which we present for the reader's convenience in {\bf Appendix}.

Note that if we view all $a_i$ as formal parameters (after replacing $\CC$ with $\CC[{\bf a}]=\CC[a_i,i\in I]$), then $\CC[{\bf a}]\otimes \gg$ is graded by the root lattice $Q:=\bigoplus\limits_{i\in I} \ZZ\alpha_i$ via $\deg a_i=\deg e_i=-\deg f_i=\alpha_i$, $i\in I$. In particular, $u_i$ is homogeneous with $\deg u_i=\alpha_i$ and both $\gg^{({\bf a})}$ and $U(\gg^{({\bf a})})$ are $Q$-graded.

Now we construct a ``real form" $\gg^{({\bf b})}$ of the electrical Lie algebra $sl_n^{({\bf b})}$.

\begin{theorem} 
\label{th:conical electrical}
Let $\gg=sl_n$. 
Then 
$g_{\bf a}u_ig_{\bf a}^{-1}=
e_i+b_{i-1}f_{i-1}$ for $i=1,\ldots,n-1$ in the notation of \eqref{eq:electric generators} (with the convention $f_0=0$, $a_0=0$), where we abbreviated $b_j:=-a_{j}a_{j+1}$ and 
$g_{\bf a}:=e^{a_{n-1}f_{n-1}}\cdots e^{a_2f_2}e^{a_1f_1}\in SL_n$.

\end{theorem}

We prove Theorem \ref{th:conical electrical} in Section \ref{subsec:proof of Theorem conical electric}.

A bit paradoxically, if all $a_i$ are imaginary hence $b_j$ in Theorem \ref{th:conical electrical}. We produce more real forms of $\gg^{(\bf a)}$ below and in Section \ref{Sect:other electrical homomorphisms}.

More generally, for any ${\bf b}=(b_1,\ldots,b_{n-2})\in \CC^{n-2}$ we  denote by $sl_n^{(\bf b)}$ the Lie subalgebra of $sl_n$ generated by $u_i:=e_i+b_{i-1}f_{i-1}$, $i=1,\ldots,n-1$ (with the convention $b_0=0$, $f_0=0$ and establish the following

\begin{theorem}  
\label{th:flat electric type A}
The Lie algebra $sl_n^{(\bf b)}$ has a presentation

$\bullet$ $[u_i,u_j]=0$ if $|i-j|>1$ 

$\bullet$ $[u_i,[u_i,u_j]]=-2b_{\min(i,j)} u_i$ if $|i-j|=1$

\end{theorem}

We  prove Theorem \ref{th:flat electric type A} in Section \ref{subsec:Proof of Theorem conical electrical2} as a particular case of Theorem \ref{th:edge electric sl_n general}.

Clearly, if all $b_i\ne 0$, Theorem \ref{th:flat electric type A} follows from Theorem \ref{th:conical electrical}. Note however, that for $n=5$, $b_1\ne 0$, $b_2=0$, $b_3\ne 0$, such a tuple $(a_1,a_2,a_3,a_4)$ does not exist and the assertion of Theorem \ref{th:flat electric type A} does not follow from Theorem \ref{th:conical electrical}.
 
Theorem \ref{th:conical electrical} suggests an ``edge model" of generalized electrical Lie algebras as follows.

Given a Kac-Moody Lie algebra $\gg$ and a tuple ${\bf b}=\{b_{ij}=b_{ji}|a_{ij}=-1\}$, we denote by $\gg^{({\bf b})}$ the Lie algebra generated by $u_i$, $i\in I$ subject to
\begin{equation}
\label{eq:deformed Serre b}  
(ad~u_i)^{1-a_{ij}}(u_j)+2\delta_{a_{ij},-1}b_{ij}u_i=0
\end{equation}
for all distinct $i,j\in I$
where $b_{ij}=0$ unless $a_{ij}=-1$.
We refer to $\gg^{({\bf b})}$ as the generalized electrical Lie algebra of the {\it edge type $\gg$} associated with the Lie algebra $\gg$ (so it is logical to think of $\gg^{({\bf a})}$ as the {\it vertex type}). Ultimately, on Theorems \ref{th:conical electrical} and \ref{th:flat electric type A}, we say that an {\it edge model} of an electric Lie algebra of type $\gg$ is any subalgebra of $\gg$ isomorphic to  $\gg^{({\bf b})}$.

Clearly, if all $b_{ij}=a_{ji}a_ia_j$ whenever $a_{ij}=-1$ in the notation of \eqref{eq:deformed Serre}, then by Theorem \ref{th:flat electric}, $\gg^{({\bf b})}\cong \gg^{({\bf a})}$. Note however, that the  generalized electrical Lie algebra $sl_n^{({\bf b})}$ of the  edge type $sl_n$ is also embedded  into $sl_n$ for any ${\bf b}\in \CC^n$ (here $b_i=b_{i,i+1}=b_{i+1,i}$ for $i=1,\ldots,n-2$). 

Based on this, we can pose a natural 

\begin{problem} Describe the set ${\mathcal V}_A$  of all tuples ${\bf b}=\{b_{ij}=b_{ji}|a_{ij}=-1\}$ such that $\gg^{(\bf b)}$ admits an edge model.   
    
\end{problem}

Theorem \ref{th:flat electric} gives an edge model for all ${\bf b}=(b_{ij})\in {\mathcal V}_A$ with 
$b_{ij}=a_{ji}a_ia_j$.

Using the edge model $sl_n^{({\bf b})}$, we generalize a remarkable observation of Lam and Pylyavskyy from \cite{LP}  that $sl_n^{({\bf 1})}\cong sp_{n-1}$, where ${\bf 1}=(1,\ldots,1)$, for all $n\ge 2$ as follows. Recall that the Lie algebra $sp_{n}$ for $n$ odd was introduced in \cite{GZ}.

\begin{theorem} 
\label{th:form}
$sl_n^{(\bf b)}$ preserves the form
$\omega_{\bf b}:=\sum\limits_{k=1}^{n-1}\left(\prod\limits_{i=k}^{n-2} (-b_i)\right )v^*_k\wedge v^*_{k+1}$ in $\CC^n=\CC v_1\oplus \cdots\oplus \CC v_n$. If $n$ is even then this form is non-degenerate as it follows from \cite{Us}. In this case $sl_n^{(\bf b)}$ is the intersection of $sp_{\omega_{\bf b}}\subset sl_n$ with the annihilator of $v^{1}=v_1-b_1v_3+b_1b_3v_5-b_1b_3b_5v_7+\cdots\in \CC^n=V_{\omega_1}$ (where $\{v_1,\ldots,v_n\}$ is the standard basis of $\CC^n$), therefore, $sl_n^{(\bf b)}$ is isomorphic to $sp_{n-1}$ when $b_1\cdots b_{n-2}\ne 0$.
If $n$ is odd, then the form $\omega_{\bf b}$ has a one-dimensional kernel that is an invariant of the action of $sl_n^{(\bf b)}$. 
\end{theorem}

We prove Theorem \ref{th:form} in Section \ref{subsec proof theorem form}.
For $n=6$ the form $\Omega_6^{(\bf b)}$ is given in $v_1,\ldots,v_6$ by its Gram matrix
$$
\Omega_6^{(\bf b)}=\tiny{\begin{pmatrix}
0 & b_1b_2b_3b_4  & 0  & 0 & 0&0 \\
-b_1b_2b_3b_4 & 0 & -b_2b_3b_4 & 0 & 0&0\\
0&b_2b_3b_4 &0& b_3b_4& 0&0\\
 0 &0&-b_{3}b_{4}&0&-b_{4}&0 \\
0 & 0&0& b_{4}&0&1\\
0& 0&0&0  & -1& 0
\end{pmatrix}.}
$$

In view of Theorem \ref{th:form}, the electrical Lie algebra of type $sl_n$ has a trivial center if $n$ is odd and a one-dimensional center if $n$ is even. It is curious that the non-injective homomorphism $sl_{2k}^{(\bf b)}\cong sp_{2k-1}\to sl_{2k-1}$ is far from the natural embedding $sp_{2k-1}\subset sl_{2k-1}$.

We will use other principle unipotents for conjugation of the standard embedding into several new ones.
\begin{example}\label{starlike}
Conjugating with  $g'_{\bf a}=e^{a_2f_2}e^{a_1f_1}e^{a_3f_3}$ and $g''_{\bf a}=e^{a_1f_1}e^{a_3f_3}e^{a_2f_2}$ in $SL_4$ respectively,  gives two new embeddings $sl_4^{(\bf a)}\hookrightarrow sl_4$:

$$
g'_{\bf a}u_i{g'}_{\bf a}^{-1}=
e_i+\delta_{i,2}(b_1f_1+b_3f_3+b_1b_3[f_1,[f_3,f_2]])\ ,
$$
$$
g''_{\bf a}u_i{g''}_{\bf a}^{-1}=
e_i+(1-\delta_{i,2})(b_if_2+b[f_i,f_2])$$
for $i=1,2,3$,
where we abbreviated $b_1=-a_1a_2$, $b_3=-a_2a_3$, $b=-a_1a_2a_3$.

\end{example}
We generalize the first embedding to all $\gg$ with a {\it conical} Dynkin diagram which includes all simply-laced $\gg$ (Theorem \ref{th:conical electric22}) and the second embedding to all $\gg$ with a star-like conical Dynkin diagram, including $D_4$ and $\hat D_4$ (Theorem \ref{th:peacock electric}). The following  result gives such an edge model in types $B$ and $C$ and is of ``conjugation type" as Theorem \ref{th:flat electric type A}.

\begin{theorem} 
\label{th:conical electrical2}

(a) Let $\gg=so_{2n+1}$ with $I=\{1,\ldots,n\}$ and the short root is $\alpha_1$. Then the assignments $u_i\mapsto  
e_i+b_{i-1,i}f_{i-1}$, $i=1,\ldots,n$
define an injective homomorphism of Lie algebras $so_{2n+1}^{(\bf b)}\hookrightarrow so_{2n+1}$ (with the convention $b_{01}=0$)

(b) Let $\gg=sp_{2n}$ with $I=\{1,\ldots,n\}$ and let the long root be $\alpha_n$. Then the assignments $u_i \mapsto e_i + b_{i-1,i} f_{i-1} - \delta_{i,n} \frac{1}{2}b_{n,n-1}^2 [f_{n-1}, [f_{n-1}, f_n]]$ for $i=1,\ldots,n$ define an injective homomorphism of the Lie algebras $sp_{2n}^{(\bf b)}\hookrightarrow sp_{2n}$.

\end{theorem}

We prove Theorem \ref{th:conical electrical2} in Section \ref{subsec:Proof of Theorem conical electrical2} by using results of 
Section \ref{Sect:other electrical homomorphisms} where we also construct vertex models for electrical Lie algebras of several other Kac-Moody types, including $E_6$, $E_7$, $E_8$, $F_4$, $G_2$ and their affine versions as well (Theorems \ref{th:conical electric22} and \ref{th:edge electric sl_n general}) along with some edge models (Theorem \ref{th:conical electric2 injective}).

Our next result also gives an edge model for dihedral, type $D$, and affine $A$ type, however, unlike Theorem \ref{th:conical electrical2}, we do not expect any intertwiner from the vertex model. 

\begin{theorem} 
\label{th:conical electrical2 non-conjugate}

(a) Let $\gg=\gg_A$ be a  Kac-Moody algebra of rank $2$ with $I=\{1,2\}$ such that $a_{21}\le -2$. Let 
$b_{12}:=a_{21}a_1a_2$. Then the assignments $u_1\mapsto e_1$, $u_2\mapsto e_2+b_{12}f_1$, define an injective homomorphism of Lie algebras $\gg^{(\bf b)}\hookrightarrow \gg$.

(b) Let $\gg=so_{2n}$, $n\ge 3$ with $I=\{1,\ldots,n\}$ and with the branch at $i=n-2$. Then the assignments $u_i\mapsto \begin{cases}
e_i+b_{i-1,i}f_{i-1} & \text{if $1\le i\le n-2$}\\
e_i+b_{n-2,i}f_{n-2}& \text{if $i\in\{n-1,n\}$}\\
\end{cases}$
$i=1,\ldots,n$,
define an injective homomorphism of Lie algebras $so_{2n}^{(\bf b)}\hookrightarrow so_{2n}$.

(c) Let $\gg=\widehat{sl}_n$, the untwisted affine Lie algebra of type $\hat A_{n-1}$ with $I=\{0,\ldots,n-1\}$. Then the assignments $u_i\mapsto e_i+b_{i-1,i}f_{i-1}$, where $i-1$ is calculated modulo $n$ for $i\in I$,
define an injective homomorphism of the Lie algebras $\widehat{sl}^{({\bf b})}\hookrightarrow \widehat{sl}_n$.

\end{theorem}

We prove Theorem \ref{th:conical electrical2 non-conjugate} in Section \ref{subsec:proof of Theorem conical electrical2 non-conjugate}.

\begin{remark} The affine electrical Lie algebra of type $\hat A_{n-1}$ from the theorem above was first discovered by T.Lam and A.Postnikov, see \cite{L}.
\end{remark}

\begin{theorem}\label{th:su}
For any ${\bf b}\in \CC^{n-1}$ one has:

(a) $sp^{({\bf b})}_{2n}\cong sl_n^{(b_1,\ldots,b_{n-2})}\ltimes J$, 
where the first factor is a copy of $sl_n^{(b_1,\ldots,b_{n-2})}$  in $sp^{({\bf b})}_{2n}$ generated by $u_1,\ldots,u_{n-1}$  and $J$ is the Lie ideal of $sp^{({\bf b})}_{2n}$ generated by $u_n$ and $sl_n^{(b_1,\ldots,b_{n-2})}$. 

As $sl_n^{(b_1,\ldots,b_{n-2})}$ module $J$ is isomorphic to $S^2V$ where $V$ is the restriction to $sl_n^{(b_1,\ldots,b_{n-2})}\subset sl_n$ of the standard $sl_n$-module $V_{\omega_1}=\CC^n$. 

(b) If all $b_i\ne 0$, then there is an isomorphism of Lie algebras $J\cong sl_{n+1}^{({\bf b'})}$ for some choice of the parameters ${\bf b'}$ (which we specify in the proof) and the Lie algebra $sp^{({\bf b})}_{2n}$ is naturally isomorphic to $sl_{n}^{(b_1,\ldots,b_{n-2})}\ltimes sl_{n+1}^{({\bf b'})}$.

\end{theorem}

\begin{remark} \label{sum} It is well-known that $\hh\ltimes \gg\cong \hh\oplus \gg$ for any Lie algebra $\gg$ and its subalgebra $\hh$ under the right adjoint action of $\hh$ on $\gg$ (namely, the diagonal copy 
of $\hh$ in $\hh\ltimes \gg$ commutes with $(0,\gg)$).

Moreover, applying this to $\gg=sp_n$, $\hh=sp_{n-1}$, Theorems \ref{th:form}  and \ref{th:su}(b) produce, after an appropriate localization, an isomorphism
$$sp^{({\bf b})}_{2n} \cong sl_{n}^{(b_1,\ldots,b_{n-2})}\oplus sl_{n+1}^{(b_1,\ldots b_{n-1})}
$$
recovering \cite[Theorem\,2.3.1]{su}\footnote{The electric algebras of type $B_n$  in~\cite{su} correspond to the type $C_n$ in our approach.}.
    
\end{remark}
We  prove Theorem \ref{th:su} in Section \ref{subsec:Proof of Theorem su}.

\begin{example} The Chevalley generators of $sp_{6}$ are given in terms of the folding of the standard generators for $sl_6$ under the flip $i\mapsto 6-i$:
$$\tilde e_1 = e_1+e_5,\,
\tilde e_2 =e_2+e_4,\,
\tilde e_3 =e_3,\,
\tilde f_1 = f_1+f_5,\,
\tilde f_2 =f_2+f_4,\,
\tilde f_3 =f_3.$$
The flip-invariant copy of $sl^{(\bf b)}_6$ is generated by the elements (Theorem \ref{th:edge electric sl_n general} with $k=3$)
$$
\tilde u_1 = e_1,\,
\tilde u_2 = e_2 + b_1f_1,\,
\tilde u_3 = \tilde e_3 + b_2\tilde f_2 - b_2^2 [\tilde f_2, [\tilde f_2, \tilde f_3]]/2$$
$$
\tilde u_4 = e_4 + b_1f_5,\,
\tilde u_5 = e_5 \ .
$$
The generators of ${sp}^{(\bf b)}_{6}$ are
$u_1 = \tilde u_1+\tilde u_5,\,
u_2 = \tilde u_2 + \tilde u_4,\,
u_3 = \tilde u_3
$.
Introduce the elements of ${sp}^{(\bf b)}_{6}$:
$w_3=u_3,\,w_2 = [u_2, [u_2, u_3]],\,
w_1 = [u_1, [u_1, w_2]]$.
These elements, on the one hand, generate the ideal $J$ and, on the other hand, a copy of $sl_{4}^{({\bf b'})}$ where ${\bf b'}=\{-32b^2_1b^2_2\,,-8b^2_2\}$.
This shows a splitting of $sp_{6}^{({\bf b})}$ into a semidirect product as in Theorem \ref{th:su}(c).

Moreover, as Remark \ref{sum} claims
there is a copy of $sl_{3}^{({\bf b'})\backslash -8b^2_2}$ that commutes with the above copy of $sl_{4}^{({\bf b'})}$. It is generated by the elements 
\[v_1 = -8b_1b_2u_1+w_1,\,v_2 = 4b_2u_2+w_2\] 
After the appropriate localization
\[sp^{(\bf b)}_{6}\cong sl_{2}^{(b_1)}\oplus sl_{3}^{({b_1,b_2})}\]
where the generators of the summands are
\[\{u_1-\frac{w_1}{8b_1b_2},\,u_2+\frac{w_2}{4b_2}\}\,\text{and}\,\{w_3,\,-\frac{w_2}{8b_2},\,\frac{w_1}{4b_1b_2}\}\]
respectively.

\end{example}

\section{Other vertex and edge models of electrical Lie algebras}
\label{Sect:other electrical homomorphisms}

In this section, we generalize Theorem \ref{th:conical electrical}, Example \ref{starlike} and Theorems \ref{th:conical electrical2},  \ref{th:conical electrical2 non-conjugate} to other semisimple and Kac-Moody algebras.

To any (generalized) Cartan matrix $A=(a_{ij},i,j\in I)$ we assign its graph $\Gamma(A)$ on the vertex set $I$ such that $(ij)$ is an edge iff $a_{ij}<0$.

We say that a  rooted tree is {\it conical} if the only possible branch is at the root (which we always denote by $0$).

For example, a linear graph with the set of vertex $[-m,n]$, $m,n\ge 0$ is a conical tree.  Also, all Dynkin diagrams of types $ADE$ are conical rooted trees.

Any rooted tree $\Gamma$ with root $0$ is naturally a layered partial order $\prec$ on $\Gamma$ with the maximal element $0$ and whose minimal elements are the leaves. Furthermore, for any $i\in \Gamma\setminus \{0\}$ denote by $i^+\in I$ the parent of $i$ in $\prec$, that is, the smallest element $j$ such that $i\prec j$. 

Likewise, if $\Gamma$ is a conical tree, for any non-leaf $i\in \Gamma\setminus \{0\}$ denote by $i^-\in I$ the only son of $i$ in $\prec$.

\begin{theorem} 
\label{th:conical electric22}
Suppose that $\Gamma(A)$ is a conical tree with the root $0$ and  $a_{ij}\in \{0,-1\}$ for all $i,j\in I$ such that $j\ne 0$. Then $(Ad~g_{\bf a})(u_i)=$
$$\begin{cases}
e_i+b_{i^-}f_{i^-}  & \text{if $i\ne 0$ }\\
e_0-a_0{\bf f}-\frac{a_0^2}{2}[{\bf f},[{\bf f},f_0]]-\sum\limits_{k\ge 3} \frac{a_0^2}{k!}(ad~{\bf f}+a_0[f_0,{\bf f}])^{k-1}([{\bf f},[{\bf f},f_0]]) & \text{if $i=0$}
\end{cases}$$
for $i\in \Gamma(A)$, where we abbreviated $b_{i^-}:=-a_ia_{i^+}$ for $i\in I$ (with the convention $b_{i^-}=0$ for any leaf $i$ of $\Gamma$), ${\bf f}:=\sum\limits_{j\in I:j^+=0} a_j f_j$
and $g_{\bf a}=e^{a_0f_0}\prod\limits_{i\ne 0} e^{a_if_i}$, where the product is decreasing (i.e., $i^+$ always precedes $i$). 

\end{theorem}

\begin{proof} We need the following 
\begin{proposition}

\label{pr:conical electric22}
Suppose that $\Gamma(A)$ is a conical tree with the root $0$ and such that $a_{ij}\in \{0,-1\}$ for all $i,j\in I$ such that  $0\notin \{i,j\}$. Then  
$$(Ad~g_{\bf a})(u_i)= \begin{cases}e_i+
b_{i^-}f_{i^-}  & \text{if $i\ne 0$ }\\
e_0+a_0[\hat {\bf f},h_0+a_0f_0]-a_0^2\sum\limits_{k\ge 2} \frac{1}{k!}(ad~\hat {\bf f})^k(f_0) & \text{if $i=0$}
\end{cases}$$
for $i\in \Gamma(A)$, where we abbreviated $b_{i^-}:=-a_ia_{i^+}$ for $i\in I$ (with the convention $b_{i^-}=0$ for any leaf $i$), ${\bf f}:=\sum\limits_{j\in I:j^+=0} a_j f_j$
and $g_{\bf a}=e^{a_0f_0}\prod\limits_{i\ne 0} e^{a_if_i}$, where the product is decreasing (i.e., $i^+$ always precedes $i$), and $\hat {\bf f}=(Ad~e^{a_0f_0})({\bf f})=\sum\limits_{k\ge 0} \frac{a_0^k}{k!}(ad~f_0)^k({\bf f})$.

\end{proposition}

\begin{proof}

Indeed, if $i\ne 0$, then $g_{\bf a}=g' g''$ where $g'=e^{a_0f_0}\prod\limits_{i\in I'} e^{a_if_i}$ and $g''=\prod\limits_{j\in I''} e^{a_jf_j}$, both products decreasing where $I''=\{j:j\preceq i\}$, $I'=(I\setminus \{0\})\setminus I''$.

As in the proof of Theorem \ref{th:conical electrical}, 
$(Ad~g'')(u_i)=e_i+b_{i^-}f_{i^-}$ hence
$$(Ad~g_{\bf a})(u_i)=(Ad~g')(e_i+b_{i^-}f_{i^-})=e_i+b_{i^-}f_{i^-}$$ because both $e_i$ and $f_{i^-}$ are fixed by $Ad~g'$.

This proves the first case.

Now let $i=0$. Then 
$$(Ad~g_{\bf a})(u_0)=(Ad~e^{a_0f_0}\prod\limits_{i\ne 0} e^{a_if_i})(u_0)$$
$$=(Ad~e^{a_0f_0}\prod\limits_{i\in I:i^+=0} e^{a_if_i})(u_0)=(Ad~e^{a_0f_0} e^{\bf f})(u_0)$$
because $[f_i,f_j]=0$ if $i^+=j^+=0$. Furthermore, 
$$(Ad~e^{\bf f})(u_0)=(Ad~e^{\bf f})(e_0+a_0h_0-a_0^2f_0)=e_0+a_0(Ad~e^{\bf f})(h_0)-a_0^2(Ad~e^{\bf f})(f_0)$$
$$=u_0+a_0[{\bf f},h_0]-a_0^2((Ad~e^{\bf f})(f_0)-f_0)$$
because 
$[{\bf f},h_0]=\sum\limits_{i\in I:i^+=0} a_ia_{0i}f_i$.
Therefore, $[{\bf f},[{\bf f},h_0]]=0$.
Finally, 
$$(Ad~g_{\bf a})(u_0)=(Ad~e^{a_0f_0})(u_0+a_0[{\bf f},h_0]-a_0^2((Ad~e^{\bf f})(f_0)-f_0))$$
$$=e_0+a_0[\hat {\bf f},h_0+2a_0f_0]-a_0^2((Ad~e^{\hat {\bf f}})(f_0)-f_0)$$
because $(Ad~e^{a_0f_0})(h_0)=h_0+2a_0f_0$.

This proves the second case.
The proposition is proved.
\end{proof}

Since $a_{0i}=-1$ whenever $i^+=0$ hence $[f_0,[f_0,{\bf f}]]=0$, we obtain
$$\hat {\bf f}={\bf f}+a_0[f_0,{\bf f}]\ .$$

Also, $[{\bf f},h_0]=-{\bf f}$ and $[[f_0,{\bf f}],h_0]=[f_0,{\bf f}]$, therefore, 
$$[\hat {\bf f},h_0+a_0f_0]=[{\bf f}+a_0[f_0,{\bf f}],h_0+a_0f_0]=-{\bf f}\ .$$

Finally, note that 
$$[\hat {\bf f},f_0]=[{\bf f}+a_0[f_0,{\bf f}],f_0]=[{\bf f},f_0]$$
and 
$$[\hat {\bf f},[{\bf f},f_0]]=[{\bf f}+a_0[f_0,{\bf f}],[{\bf f},f_0]]]=[{\bf f},[{\bf f},f_0]]]$$

The theorem is proved.
\end{proof}

Theorem \ref{th:conical electric22} covers the cases when $\gg$ is of types $E_6,E_7,E_8$, and $F_4$.

Viewing the Dynkin diagram of $sl_n$ as a conical tree with a ``root" $0=k\in\{1,\ldots,n-1\}$ we obtain the following corollary of Theorem \ref{th:conical electric22}, which generalizes Theorem \ref{th:conical electrical}.

\begin{theorem}
 \label{th:edge electric sl_n general}   
For any $k\in\{1,\ldots,n-1\}$ the assignments 
$$u_i\mapsto 
\begin{cases}
e_i+b_{i-1}f_{i-1}  & \text{if $i<k$}\\
e_i+b_if_{i+1}  & \text{if $i>k$}\\
e_k+b_{k-1}f_{k-1}+b_kf_{k+1}+b_{k-1}b_k[f_{k-1},[f_{k+1},f_k]] & \text{if $i=k$}
\end{cases}
$$
define an injective homomorphism $sl_n^{({\bf b})}\hookrightarrow sl_n$ (where  $b_i:=-a_ia_{i+1}$).

\end{theorem} 

We prove Theorem \ref{th:edge electric sl_n general} in Section \ref{subsec:Proof of Theorem conical electrical2}.

The following is an immediate corollary of Theorem \ref{th:conical electric22}.

\begin{corollary}
\label{cor:edge electric conical star}    
Let $\Gamma(A)$ be a conical star tree (i.e., every non-root is a leaf) with $a_{i0}=a_{0i}=-1$ for all $i\in I\setminus \{0\}$. Then 
the assignments
$$u_i\mapsto  
e_i-a_0\delta_{i,0}({\bf f}-\frac{a_0}{2}[{\bf f},[{\bf f},f_0]]-a_0\sum\limits_{k\ge 3} \frac{1}{k!}(ad~{\bf f}+a_0[f_0,{\bf f}])^{k-1}([{\bf f},[{\bf f},f_0]])) \ ,$$
where  ${\bf f}:=\sum\limits_{j\ne 0} a_j f_j$, define an injective homomorphism $\gg^{({\bf a})}\hookrightarrow \gg$.

\end{corollary}

\begin{theorem} 
\label{th:peacock electric}
Suppose that $\Gamma(A)$ is a conical tree with the root $0$, $J_+$ is 
 a set of leaves of $\Gamma(A)$ attached to the $0$,  and  $a_{ij}\in \{0,-1\}$ for all $i,j\in I\setminus J_+$ such that $j\ne 0$ and (i.e., $J_+^+=\{0\}$) such that $a_{i0}=-1$ for all $i\in J_+$. Then
$$(Ad~g_{\bf a})(u_i)=\begin{cases}

e_i+b_if'_i-\frac{b_i^2}{2}[[f_i,f'_i],f'_i]\\
-a_i^2\sum\limits_{k\ge 3}\frac{a_0^k}{k!}(ad~(f'_i+a_i[f_i,f'_i]))^{k-2}([[f_i,f'_i],f'_i])  & \text{if $i\in J_+$}\\
e_0-a_0{\bf f}-\frac{a_0^2}{2}[{\bf f},[{\bf f},f'_0]]\\
-a_0^2\sum\limits_{k\ge 3} \frac{1}{k!}  ((ad~{\bf f}+a_0[f'_0,{\bf f}])^{k-1}([{\bf f},[{\bf f},f'_0]]))
& \text{if $i=0$}\\
e_i+b_{i^-}f_{i^-}  & \text{otherwise}\\
\end{cases}$$
for $i\in \Gamma(A)$, where we abbreviated $b_{i^-}:=-a_ia_{i^+}$ for $i\in I\setminus J_+$, $b_i:=-a_0a_i$, $i\in J_+$ (with the convention $b_{i^-}=0$ for any leaf $i$), ${\bf f}_+:=\sum\limits_{j\in J_+} a_j f_j$,
${\bf f}:=\sum\limits_{j\in I\setminus J_+:j^+=0} a_j f_j$, $f'_0:=\sum\limits_{k\ge 0}\frac{1}{k!} (ad~{\bf f}_+)^k(f_0)$, $f'_i:=\sum\limits_{k\ge 0}\frac{1}{k!}(ad~({\bf f}_+-a_if_i))^k(f_0)$, $i\in J_+$,
and $g_{\bf a}= e^{{\bf f}_+} 
e^{a_0f_0}\prod\limits_{i\in I\setminus (\{0\}\cup J_+)} e^{a_if_i}$, where the product is decreasing (i.e., $i^+$ always precedes $i$). 

\end{theorem}

\begin{proof} The case $i\ne 0$, $i\notin J_+$ follows from the proof of Theorem \ref{th:conical electric22}.

Now let $i\in J_+$. Note that $(Ad~e^{a_jf_j})(u_i)=u_i$ for all $j\ne 0$, $j\notin J_+$. Also,  $i\in J_+$:
$$(Ad~e^{a_if_i}e^{a_0f_0})(u_i)=(Ad~e^{a_if_i}e^{a_0f_0}e^{-a_if_i})(e_i)=(Ad~e^{a_0(f_0+a_i[f_i,f_0])})(e_i)$$
Taking into account that 
$$(ad~((f_0+a_i[f_i,f_0]))(e_i)=a_i[[f_i,f_0],e_i]$$
$$=a_i[[f_i,e_i],f_0]=-a_0a_i[h_i,f_0]=-a_if_0$$ and
$$(ad~(f_0+a_i[f_i,f_0]))^{k-1}(f_0)=a_i[[f_i,f_0],f_0]]\ ,$$
we obtain
$$(Ad~e^{a_if_i}e^{a_0f_0})(u_i)=e_i-a_0a_if_0+\sum_{k\ge 2}\frac{a_0^{k-1}}{k!}(ad~(f_0+a_i[f_i,f_0]))^{k-1}(a_{i0}a_0a_if_0)$$
$$=e_i+b_if_0-\frac{b_i^2}{2}[[f_i,f_0],f_0]-a_0a_i^2\sum_{k\ge 3}\frac{a_0^{k-1}}{k!}(ad~(f_0+a_i[f_i,f_0]))^{k-2}([[f_i,f_0],f_0]) \ .$$

Finally,
$$Ad~g_{\bf a}(u_i)=(Ad~e^{{\bf f}_+}e^{a_0f_0}) (u_i)=(Ad~e^{{\bf f}_+-a_if_i})((Ad~e^{a_if_i}e^{a_0f_0})(u_i))$$
$$e_i+b_if'_{0i}-\frac{b_i^2}{2}[[f_i,f'_{0i}],f'_{0i}]-a_i^2\sum_{k\ge 3}\frac{a_0^k}{k!}(ad~(f'_{0i}+a_i[f_i,f'_{0i}]))^{k-2}([[f_i,f'_{0i}],f'_{0i}]) \ .$$

Now consider the remaining case $i=0$. Then
$$Ad~g_{\bf a}(u_0)=(Ad~e^{{\bf f}_+}e^{a_0f_0} e^{\bf f}e^{-a_0f_0})(e_0)=(Ad~e^{{\bf f}_+}e^{\hat {\bf f}})(e_0)=Ad~e^{{\bf f}_+}(e^{\bf f}(e_0))
$$
where $\hat {\bf f}=(Ad~e^{a_0f_0})({\bf f})={\bf f}+a_0[f_0,{\bf f}]$.

Finally, using the argument from the proof of Theorem \ref{th:conical electric22} for $i=0$ and replacing $f_0$ with $f'_0=Ad~e^{{\bf f}_+}(f_0)$, we finish the proof of Theorem \ref{th:peacock electric}.

\end{proof}

The following is an immediate corollary of Theorem \ref{th:peacock electric}.

\begin{corollary}
  \label{cor:peacock electric2}  

 In the assumptions of Theorem \ref{th:peacock electric} assume additionally, that $I\setminus J_+$ is of type $A$. Then 
 
 (a) the assignments
$$u_i\mapsto \begin{cases}
e_i+b_{i^-}f_{i^-}  & \text{if $i\notin J_+$ }\\
e_i+b_if_0+b_i\sum\limits_{k\ge 1} \frac{1}{k!}(ad~({\bf f}_+-a_if_i))^k(f_0)  & \text{if $i\in J_+$}\\
\end{cases}$$
define an injective homomorphism $\gg^{({\bf a})}\hookrightarrow \gg$.

(b) 
Suppose  that  $I=\{1,\ldots,n\}$, $\{1,\ldots,n-1\}$ is of type $A_{n-1}$, and $a_{in}=a_{ni}=\delta_{ik}$.
Then the assignments 
$$u_i\mapsto 
\begin{cases}
e_i+b_{i+1}f_{i+1}  & \text{if $1\le i<k$}\\
e_k+b_{k-1,k}f_{k-1}+b_{k,k+1}f_{k+1}\\
+b_{k-1,k}b_{k,k+1}[f_{k-1},[f_{k+1},f_k+a_n[f_n,f_k]]] & \text{if $i=k$}\\
\end{cases}
$$
where $b_{ij}:=a_{ij}a_ia_j$,
define an injective homomorphism $\gg^{({\bf a})}\hookrightarrow \gg$.

\end{corollary}

\begin{theorem} 
\label{th:conical electric2}
Let $\gg=\gg_A$ be a  Kac-Moody algebra with $I=\{1,\ldots,r\}$ such that the Cartan matrix $A$ of $\gg$ satisfies $a_{ij}a_{ji}\ne 0$ iff $|i-j|=1$ and $a_{ij}\le a_{i-1,j-1}$ for all distinct $i,j=2,\ldots,r$. 

Suppose also that the following three conditions hold.

$\bullet$ Either $a_{12}=a_{21}=-1$ or $a_{21}<-1$.

$\bullet$ Either $a_{i-1,i}a_{i,i+1}>1$ or $a_{i-1,i}=a_{i,i+1}=a_{i,i-1}=-1$  for $i=2,\ldots,r$.

$\bullet$ Either 
$a_{i,i-1}a_{i+1,i}>1$ or
$a_{i+1,i}=a_{i,i-1}=a_{i,i+1}=-1$ for $i=2,\ldots,r$. 

Then the assignments $u_i\mapsto \begin{cases}
e_1 & \text{if $i=1$}\\
e_i+b_{i-1}f_{i-1} & \text{if $2\le i\le r$}\\
\end{cases}$
define a homomorphism of Lie algebras ${\gg'}^{(\bf b)}\to \gg$, where we abbreviated $b_j:=a_{j+1,j}a_ja_{j+1}$ for $j=1,\ldots,r-1$, similarly to Theorem \ref{th:conical electrical2} and $A'$ is $I\times I$ the Cartan matrix of $\gg'$ given by $a'_{12}=a_{12}$ $a'_{21}=a_{21}$ and
$a'_{ij}=
\min(a_{i-1,j-1},a_{ij}) 
$ for $i,j=2,\ldots,r$

\end{theorem}

\begin{theorem} 
\label{th:conical electric3} 
Suppose that $I=\{1,\ldots,r+d\}$, $r\ge 1$, $d\ge 1$ such that $a_{ij}a_{ji}\ne 0$ for distinct $i,j$ iff either $i,j\le r$ and $|i-j|=1$ or $\min(i,j)=r$, the restriction of $A$ to  $\{1,\ldots,r,r+i\}$ satisfies the assumptions of Theorem \ref{th:conical electric2} for $i=1,\ldots,d$ and $a_{r+i,r+j} =0$ for all distinct $i,j=1,\ldots,d$. 
Then the assignments $$u_i\mapsto \begin{cases}
e_1 & \text{if $i=1$}\\
e_i+b_{i-1}f_{i-1} & \text{if $2\le i\le r$}\\
e_i+b_if_r& \text{if $r< i\le r+d$}\\
\end{cases}$$
define a  homomorphism of Lie algebras ${\gg''}^{(\bf b)}\to \gg$, where we abbreviated $b_i:=a_{ir}a_ra_i$ for $i=r+1,\ldots,r+d$ and $A''$ is the $I\times I$ Cartan matrix of $\gg''$ with $a''_{ij}=a'_{ij}$ whenever $i,j=1,\ldots,r$, $a''_{ij}=\min(a_{i-1,j-1},a_{ij})$ whenever $\min(i,j)=r$, $i\ne j$ and the remaining $a''_{ij}=0$.

\end{theorem}

We prove Theorems \ref{th:conical electric2} and \ref{th:conical electric3} in Section \ref{subsec:proof of Theorem conical electric2}.

Taking $\gg''=\gg$ in Theorem \ref{th:conical electric3}, we obtain the following

\begin{theorem}
\label{th:conical electric2 injective}  
Suppose that $I=\{1,\ldots,r+d\}$, $a_{ij}a_{ji}\ne 0$ for distinct $i,j$ iff either $i,j\le r$ and $|i-j|=1$ or $\min(i,j)=r$ (as in Theorem \ref{th:conical electric3})  
and the Cartan matrix of $\gg$ satisfies

$\bullet$ Either $a_{12}=a_{21}=-1$ or $a_{21}<-1$.

$\bullet$ Either $a_{i-1,i}=a_{i,i+1}=a_{i,i-1}=-1$ or $a_{i,i+1}<-1$ and $a_{i,i+1}\le a_{i-1,i}$ for $i=2,\ldots,r$

$\bullet$ Either $a_{i+1,i}=a_{i,i-1}=a_{i,i+1}=-1$ or $a_{i+1,i}<-1$ and $a_{i+1,i}\le a_{i,i-1}$ for $i=2,\ldots,r$.

$\bullet$ $a_{ij}\le a_{i-1,j-1}$ whenever $\min(i,j)=r$, $i\ne j$.

Then the subalgebra  generated by $u_i:= \begin{cases}
e_1 & \text{if $i=1$}\\
e_i+b_{i-1}f_{i-1} & \text{if $2\le i\le r$}\\
e_i+b_if_r& \text{if $r< i\le r+d$}\\
\end{cases}$ is naturally isomorphic to 
$\gg^{(\bf b)}$ for any ${\bf b}=(b_1,\ldots,b_{n-1})$.

\end{theorem}

We prove Theorem \ref{th:conical electric2 injective} in Section \ref{subsec:proof of Theorem conical electric2 injective}.

As in the proof of Theorem \ref{th:flat electric}, all results of this section are valid if one replaces $\CC$ with and $\CC$-algebra containing all $a_i$ (resp. all $b_{ij}$).

\section{Proof of main results}

\subsection{Proof of Theorem\ref{th:canonical electric}} 
\label{subsec:proof of Theorem canonical electric}
Let
$e_{i^rj}:=\frac{1}{r!}(ad~e_i)^r(e_j)$, $f_{i^rj}:=\frac{1}{r!}(ad~f_i)^r(f_j)$, and $u_{i^rj}:=\sum\limits_{k=0}^r a_i^{r-k} \binom{a_{ij}+r-1}{r-k}(e_{i^k j}-(-1)^ra_i^{2k}a_j^2f_{i^k j})$
for any distinct $i,j\in I$ and $r\ge 1$.

We need the following result.

\begin{proposition}
\label{pr:iterated Serre}
 $[u_i,u_{i^r j}]=(r+1)u_{i^{r+1} j}$ for all distinct $i,j\in I$ and $r\ge 1$.
\end{proposition}

\begin{proof} The following is immediate.

\begin{lemma}
$[f_i,e_{i^k j}]=-(a_{ij}+k-1)e_{i^{k-1} j}$, $[e_i,f_{i^k j}]=-(a_{ij}+k-1)f_{i^{k-1} j}$ for all $k\ge 0$ (with the convention $f_{i^{-1}j}=e_{i^{-1}j}=0$).
\end{lemma}

Thus, we have for $r\ge 1$:
$$[e_i,u_{i^rj}]=\sum_{k=0}^r a_i^{r-k} \binom{a_{ij}+r-1}{r-k}([e_i,e_{i^k j}]-(-1)^ra_i^{2k}a_j^2[e_i,f_{i^k j}])$$
$$=\sum_{k=0}^r a_i^{r-k} \binom{a_{ij}+r-1}{r-k}((k+1)e_{i^{k+1} j}+(-1)^r(a_{ij}+k-1)a_i^{2k}a_j^2f_{i^{k-1} j}) \ ,$$
$$[h_i,u_{i^rj}]=\sum_{k=0}^r a_i^{r-k} \binom{a_{ij}+r-1}{r-k}([h_i,e_{i^k j}]-(-1)^ra_i^{2k}a_j^2[h_i,f_{i^k j}])$$
$$=\sum_{k=0}^r a_i^{r-k} \binom{a_{ij}+r-1}{r-k}(2k+a_{ij})(e_{i^k j}+(-1)^ra_i^{2k}a_j^2f_{i^k j})\ ,$$
$$[f_i,u_{i^rj}]=\sum_{k=0}^r a_i^{r-k} \binom{a_{ij}+r-1}{r-k}([f_i,e_{i^k j}]-(-1)^ra_i^{2k}a_j^2[f_i,f_{i^k j}])$$
$$=\sum_{k=0}^r a_i^{r-k} \binom{a_{ij}+r-1}{r-k}(-(a_{ij}+k-1)e_{i^{k-1} j}-(-1)^r(k+1)a_i^{2k}a_j^2f_{i^{k+1} j})$$
Therefore, 
$$[u_i,u_{i^rj}]=\sum_{k=0}^{r+1}c_k a_i^{r+1-k}(e_{i^kj}+(-1)^ra_i^{2k}a_j^2f_{i^kj})$$
where $c_{r+1}=r+1$ and
$$c_k=(r-y)\binom{x}{y+1}+(x+r+1-2y)\binom{x}{y}+ (x-y+1)\binom{x}{y-1}$$
$$=(r+1)\binom{x+1}{y+1}-(y+1)\binom{x}{y+1}+(x-2y)\binom{x}{y}+y\binom{x}{y}=(r+1)\binom{x+1}{y+1}$$
because $\binom{x}{y+1}+\binom{x}{y}=\binom{x+1}{y+1}$,  $(y+1)\binom{x}{y+1}=(x-y)\binom{x}{y}$, $(x-y+1)\binom{x}{y-1}=y\binom{x}{y}$,
where we abbreviated $x:=a_{ij}+r-1$, $y:=r-k$.

This proves Proposition \ref{pr:iterated Serre}.
\end{proof}

Furthermore,
$$[u_i,u_j]=[e_i+a_ih_i-a_i^2f_i,e_j+a_jh_j-a_j^2f_j]=e_{ij}+a_i^2a_j^2f_{ij}$$
$$-a_ja_{ji}e_i+a_ia_{ij}e_j+a_j^2a_ia_{ij}f_j-a_i^2a_ja_{ji}f_i=u_{ij}-a_ja_{ji}(e_i+a_i^2f_i)$$

Taking into account that
$$[u_i,e_i+a_i^2f_i]=[e_i+a_ih_i-a_i^2f_i,e_i+a_i^2f_i]=2a_iu_i$$ and
using Proposition \ref{pr:iterated Serre} with $r=2$, we obtain
$$(ad~u_i)^2(u_j)=2u_{i^2j}-2a_ia_ja_{ji}u_i\ .$$
In turn, this and Proposition \ref{pr:iterated Serre}  imply by an induction in $r\ge 3$ that 
$$(ad~u_i)^r(u_j)=r!u_{i^rj}$$
for all $r\ge 3$.

Finally, if $r=1-a_{ij}$ then $\binom{a_{ij}+r-1}{r-k}=0$ for all $k<r$ and $e_{i^r j}= f_{i^r j}=0$ which implies that $u_{i^r j}=0$.

Theorem \ref{th:canonical electric} is proved.
\endproof

\subsection{Proof of Theorem \ref{th:flat electric}}
\label{subsec:proof of Theorem flat electric}

We apply the results of {\bf Appendix} to our situation as follows.

Let $R$ be any $\CC$-algebra, fix a family ${\bf a}=(a_i,i\in I)$ of elements of $R$.

Denote  $A:=R\otimes U(\gg)$. By definition,   the subalgebra $B$ of $A$ generated by $X:=\{u_i=e_i+a_ih_i-a_i^2f_i,i\in I\}$  over $R$ is $U(\gg^{(\bf a)})$. 
According to \eqref{eq:deformed Serre}, the
kernel of the canonical  homomorphism $\varphi_{X,B}:R\langle X\rangle \twoheadrightarrow B$  contains the set $Y:=\{(ad~u_i)^{1-a_{ij}}(u_j)+2\delta_{a_{ij},-1}a_{ji}a_ia_ju_i:i,j\in I,i\ne j\}$.

Let us define a filtration on $A$ by setting $\deg e_i=1$, $\deg f_i=\deg h_i=0$ for $i\in I$. Clearly, $gr A$ is generated by $e_i,f_i,h_i$ subject to the usual Serre relations except for $[e_i,f_i]=h_i$ which is replaced by $[e_i,f_i]=0$. In particular, the subalgebra $B_0$ of $gr A$ generated by $e_i,i\in I$ is naturally isomorphic to $R\otimes U(\nn)$.

Then the natural inclusion of $R$-algebras $B\subset A$ is compatible with the filtration and defines a natural inclusion of graded algebras $grB\subset B_0$ because $gr~u_i=e_i$ for $i\in I$ (that is,  $grX=\{e_i,i\in I\}$).

Clearly, $gr~Y=\{(ad~e_i)^{1-a_{ij}}(e_j):i,j\in I,i\ne j\}$ and  is the presentation of $B_0$, i.e., $gr~Y$ generates the kernel of the  canonical homomorphism $\varphi_{grX,gr B}:R\langle grX\rangle \twoheadrightarrow grB=B_0$.

Thus, all conditions of Proposition \ref{pr:flat filtred algebras} are met. The proposition guarantees that the relations \eqref{eq:deformed Serre} give a presentation of $B=U(\gg^{(\bf a)})$ as an associative algebra over $R$ hence of the Lie algebra $\gg^{(\bf a)}$ over $R$.

The theorem is proved.
\endproof

\subsection{Proof of Theorem \ref{th:conical electrical}}
\label{subsec:proof of Theorem conical electric}
Clearly, $u_i=(Ad~e^{-a_if_i})(e_i)$ for all $i\in I=\{1,\ldots,n-1\}$. 

Then for $i=1$ we have
$$(Ad~g_{\bf a})(u_1)=(Ad~e^{a_{n-1}f_{n-1}}\cdots e^{a_2f_2})(e_1)=e_1\ .$$

Now let $i\ge 2$. 
Then 
$$(Ad~e^{a_if_i}\cdots e^{a_2f_2}e^{a_1f_1})(u_i)=(Ad~e^{a_if_i}e^{a_{i-1}f_{i-1}}e^{-a_if_i})(e_i)$$
$$=(Ad~e^{a_{i-1}(f_{i-1}+a_i[f_i,f_{i-1}])})(e_i)=(Ad~e^{a_{i-1}a_i[f_i,f_{i-1}]})(e_i)$$
$$=e_i+a_{i-1}a_i[[f_i,f_{i-1}],e_i]=e_i-a_{i-1}a_if_{i-1}\ .$$
Here we used the facts that $(Ad~e^{af_j})(e_i)=e_i$ if $i\ne j$, $(Ad~e^{af_j})(f_i)=f_i$ whenever $|i-j|>1$ and
$(Ad~e^{a_if_i})(f_{i-1})=f_{i-1}+a_i[f_i,f_{i-1}]$ because $a_{i,i-1}=-1$.  
This is based on the following general fact
\begin{equation}\label{eq:Ad-ad}
    (Ad~e^x)(y)=e^{ad~x}(y)\ .
\end{equation}

Finally, 
$$(Ad~g_{\bf a})(u_i)=(Ad~e^{a_{n-1}f_{n-1}}\cdots e^{a_{i+1}f_{i+1}})(e_i-a_{i-1}a_if_{i-1})=e_i-a_{i-1}a_if_{i-1}\ .$$

The theorem is proved. \endproof

\subsection{Proof of Theorem \ref{th:form}}
\label{subsec proof theorem form} 

Recall that $sl_n$ acts on $V^*$ via $$e_i(v_j^*)=\delta_{ij}v^*_{i+1},~f_i(v_j^*)=\delta_{i+1,j}v^*_i\,$$
where $\{v^*_i\}$ is the basis of $V^*$ dual to $\{v_i\}$.
Therefore, a direct calculation shows that the generators $u_i=e_i+
b_{i-1}f_{i-1}$
for $i=1,\ldots,n-1$ of  $sl_n^{({\bf b})}$
preserve 
\bea
\omega_{\bf b}:=\sum\limits_{k=1}^{n-1}\left(\prod\limits_{i=k}^{n-2} (-b_i)\right )v^*_k\wedge v^*_{k+1}\eea 

It is immediate that $v^1\in V$ is invariant under the restriction from $sl_n$-action on $V:=V_{\omega_1}$ to the $sl_n^{({\bf b})}$-action on $V$. This, in particular, implies that the orthogonal complement $V^1\subset V^*$ to $v^1$ is a (co-dimension $1$) $sl_n^{({\bf b})}$-submodule of $V^*$.

It is easy to see that $\omega_{\bf b}\in \Lambda^2 V^*$ in fact belongs to $\Lambda^2 V^1$ for odd $n$. Clearly, in this case vectors $w_i=\begin{cases}
b_iv_i^*+v_{i+2}^* &\text{if $i$ is odd}\\
v_i^* &\text{if $i$ is even}\\
\end{cases}$
$i=1,\ldots,n-1$ form a basis of $V^1$ and
$-\omega_{\bf b}=w_1\wedge w_2+w_3\wedge w_4\cdots w_{n-2}\wedge w_{n-1}$
and it of rank $n-1$ if $b_1\cdots b_{n-2}\ne 0$. Since $\dim sl_n^{({\bf b})}=\dim sp_{n-1}$, this implies that $sl_n^{({\bf b})}\cong sp_{n-1}$, and thus finishes the proof for odd $n$.

Let $n$ be even.
It well-known that $\det M=\prod\limits_{i=1}^k m_{2i,2i-1}m_{2i-1,2i}$ for any tridiagonal $2k\times 2k$-matrix $M$ with $m_{ii}=0$. 

This implies that $\omega_{\bf b}$ is of rank $n$  if $b_1\cdots b_{n-2}\ne 0$. Therefore, in this case, the annihilator of $\omega_{\bf b}$ in $sl_n$ is isomorphic to $sp_n$. Therefore, $sl_n^{({\bf b})}$ identifies with a Lie subalgebra of $sp_n$.

Since $V$ is a simple $sp_n$-module, then for any nonzero $v\in V$ the annihilator of $v$ is isomorphic to the maximal parabolic subalgebra of $sp_n$, i.e., to $sp_{n-1}$. 

Since $sl_n^{(b)}$ also annihilates $v^1\in V$, $sl_n^{(b)}$ is a Lie subalgebra of $sp_{n-1}$. Then the dimension argument once again implies that 
$sl_n^{({\bf b})}\cong sp_{n-1}$ as well for even $n$.

The theorem is proved. \endproof

\subsection{Proof of Theorems  \ref{th:flat electric type A}, \ref{th:edge electric sl_n general}, and \ref{th:conical electrical2}} 

\label{subsec:Proof of Theorem conical electrical2}

Note that Theorem \ref{th:flat electric type A} is a particular case of Theorem \ref{th:edge electric sl_n general}(a).

Let $\hat {\hat R}=\CC[a_i,i\in I]$ and let $\hat R$ be its subalgebra generated by all $b_{ij}:=a_{ji}a_ia_j$ whenever $a_{ij}=-1$ in the assumptions of Theorems \ref{th:conical electrical2} and \ref{th:edge electric sl_n general}. Clearly, $\hat R=\CC[b_{ij}]$.

Indeed, let $u_i=u_i^{\bf b}\in \hat R\otimes \gg$, $i\in I$ be the elements as in (the right hand sides of) Theorems \ref{th:edge electric sl_n general}   and \ref{th:conical electrical2}. Let $v_{ij}^{\bf b}\in \hat R\otimes \gg$ be the left hand side of \eqref{eq:deformed Serre b}  in each case. 

\begin{lemma} 
\label{le:b-relations}
$v_{ij}^{\bf b}=0$ in each case of Theorems \ref{th:conical electrical2} and \ref{th:edge electric sl_n general}.
    
\end{lemma}

\begin{proof} Theorem \ref{th:conical electric22} taken in $\hat {\hat R}\otimes sl_n$ and $0=k\in A_{n-1}$ (so that $A_{n-1}$ is the rooted conical tree with sub-trees $\{1,\ldots,k-1\}$ and $\{k+1,\ldots,n-1\}$) implies 
Lemma \ref{le:b-relations} in the assumptions of Theorem   \ref{th:edge electric sl_n general} for $\hat R\otimes sl_n$.

Likewise, Theorem \ref{th:conical electric22} taken in $\hat {\hat R}\otimes so_{2n+1}$ and $0=n\in B_n$ (so that $B_n$ is the rooted conical tree with a single sub-tree $\{1,\ldots,k-1\}$) implies Lemma \ref{le:b-relations}
in the assumptions of Theorem   \ref{th:conical electrical2}(a) in $\hat R\otimes so_{2n+1}$.  

Finally, Theorem \ref{th:conical electric22} taken in $\hat {\hat R}\otimes sp_{2n}$ and $0=n-1\in C_n$, $J_+=\{n\}$ (so that $C_n$ is the rooted conical tree with a single sub-tree $\{1,\ldots,n-2\}$ and $J_+$) implies Lemma \ref{le:b-relations}
in the assumptions of Theorem   \ref{th:conical electrical2}(b) in $\hat R\otimes sp_{2n}$ (with $b_i=b_{i-1,i}$ for $i<n$ and $b_n=b_{n,n-1}$).

The lemma is proved.
\end{proof}

Furthermore, 
we proceed similarly to the proof of Theorem \ref{th:flat electric}.   Indeed, let $R$ be any $\CC$-algebra containing all $b_{ij}$.
 
Denote  $A:=R\otimes U(\gg)$. By definition, the subalgebra $B$ of $A$ generated by $X:=\{u_i,i\in I\}$  over $R$ is $U(\gg^{(\bf b)})$.

 The canonical specialization homomorphism $\hat R\twoheadrightarrow R$ implies that Lemma \ref{le:b-relations} holds over $R$ as well, that is,  $u_i$ satisfy the relations of Theorems \ref{th:conical electrical2} and \ref{th:edge electric sl_n general}, that is, in the notation of Proposition \ref{pr:flat filtred algebras},
the
kernel of the canonical  homomorphism $\varphi_{X,B}:R\langle X\rangle \twoheadrightarrow B$  contains the set $Y:=\{v_{ij}^{\bf b},i,j\in I\} \ .
$

Let us define a filtration on $A$ same way as in the proof of Theorem \ref{th:flat electric}.

In particular, the subalgebra $B_0$ of $gr~A$ generated by $e_i,i=1,\ldots,n-1$ is naturally isomorphic to $R\otimes U(\nn)$, where $\nn=<e_i,i\in I>$ is the nilpotent Lie subalgebra of $\gg$.

Then the natural inclusion of $R$-algebras $B\subset A$ is compatible with the filtration and defines a natural inclusion of graded algebras $grB\subset B_0$ because $gr\,u_i=e_i$ for $i\in I$ (that is,  $grX=\{e_i,i\in I\}$).

Clearly, $grY=\{(ad~e_i)^{1-a_{ij}}(e_j):i\ne j\}$ and  is the presentation of $B_0$, i.e., $grY$ generates the kernel of the  canonical homomorphism $\varphi_{grX,gr B}:R\langle grX\rangle \twoheadrightarrow grB=B_0$.

Thus, all conditions of Proposition \ref{pr:flat filtred algebras} are met. The proposition guarantees that the relations of Theorems \ref{th:conical electrical2} and \ref{th:edge electric sl_n general} give a presentation of $B=U(\gg^{(\bf b)})$ as an associative algebra over $R$ hence, of the Lie algebra $\gg^{(\bf b)}$ over $R$.

Theorems \ref{th:flat electric type A}, \ref{th:edge electric sl_n general}, and \ref{th:conical electrical2} are proved. \endproof

\medskip

\subsection{Proof of Theorem \ref{th:su}}
\label{subsec:Proof of Theorem su}
\begin{proof}
The generators of $sp^{({\bf b})}_{2n}$ are as in Theorem \ref{th:conical electrical2}
$$u_i =\tilde e_i + b_{i-1} \tilde f_{i-1} - \delta_{i,n} \frac{1}{2}b_{n-1}^2 [\tilde f_{n-1}, [\tilde f_{n-1}, \tilde f_n]]$$
(a) It is clear that the set of elements 
\[\{u_1,\ldots ,u_{n-1}\}\]
generates a subalgebra of $sp^{({\bf b})}_{2n}$ isomorphic to $sl_n^{(b_1,\ldots ,b_{n-2})}$. Define $J$ as an ideal of $sp^{({\bf b})}_{2n}$ generated by the set
$[sl_n^{(b_1,\dots ,b_{n-2})},u_{n}]$.

Recall that ${sp}_{2n}$ contains $sl_n$ and splits as a $sl_n$-module
$$sp_{2n}\cong V_1 \oplus sl_n\oplus V_{2}$$
where the $sl_n$ modules $V_{1}$ and $V_2$ are $\CC e_n+[sl_n,e_n]$ and $\CC f_n+[sl_n,f_n]$ respectively. Note that
$V_1=S^2 V_{\omega_1}$. As a direct calculation shows the projection of $J$ to $V_1$ is an isomorphism of vector spaces and moreover it is an $sl^{b_1,\ldots b_{n-2}}_n$ module isomorphism where $V_1$
is a module over $sl^{b_1,\ldots b_{n-2}}_n$ since $sl^{b_1,\ldots b_{n-2}}_n\subset sl_n$.

(b) Introduce the set of elements
$u'_i$, $1\leq i\leq n$, as follows:
$$
u'_n=u_n,\,\,u'_{n-1} = [u_{n-1},[u_{n-1},u_n]]
$$
and
$
u'_{n-k} = Ad^2_{u_{n-k}}\circ Ad^2_{u_{n-k+1}}\circ \dots \circ Ad^2_{u_{n-1}}(u_n)
$.
There is an isomorphism of Lie algebras $\phi:sl_{n+1}^{(\bf b')}\rightarrow J$ such that $\phi(v_i)=u'_i$, where $v_i$ are the generators of $sl_{n+1}^{({\bf b'})}$, ${\bf b'}=(b'_1,\ldots,b'_{n-1})$ where
$b'_k=-2^{2(n-k)+1}\prod\limits_{i=k}^{n-1} b^2_i$
$k=1,\ldots n-1$.

Moreover $[u'_{n-k},u'_{n-k\pm 1}]=(-2)^{k} \left(\prod\limits_{i=1}^{k-1} b_{n-i}\right)[u_{n-k},u'_{n-k\pm 1}]$ for $1\leq i< n$  hence 
$[u'_{n-k}+(-2)^{k+1} \left(\prod\limits_{i=1}^k b_{n-i}\right)u_{n-k},u'_{n-k\pm 1}]=0$
Observe finally, that the elements
$$u''_{n-k}=u'_{n-k}+(-2)^{k+1} \left(\prod_{i=1}^k b_{n-i}\right)u_{n-k}$$
generate a copy of $sl_{n}^{(b_1,\ldots,b_{n-2})}$ after the appropriate localization.

This finishes the proof.

\end{proof}

\subsection{Proof of Theorems \ref{th:conical electric2} and \ref{th:conical electric3}} 
\label{subsec:proof of Theorem conical electric2} 

We need the following fact. 

\begin{proposition}
\label{pr:local relations}

Let $\gg$ be a Kac-Moody Lie algebra with $I=\{1,\ldots,r\}$ and  $u_i:=e_i+b_{i-1}f_{i-1}$ (with $b_0=0$). Then for $k\ge 2$ one has: 

(a) $(ad~u_1)^k(u_2)=(ad~e_1)^k(e_2)-2\delta_{k,2}b_1u_1$,

\hskip 6.5mm $(ad~u_2)^k(u_1)=(ad~e_2)^k(e_1)-2b_1\delta_{k,2}(u_2+(a_{12}+1)e_2)$.

(b) $(ad~u_i)^k(u_{i+1})=(ad~e_i)^k(e_{i+1})+b_{i-1}^kb_i(ad~f_{i-1})^k(f_i)
-2b_i\delta_{k,2}(u_i-b_{i-1}(a_{i,i-1}+1)f_{i-1})$,

\hskip 6.5mm
$(ad~u_{i+1})^k(u_i)=(ad~e_{i+1})^k(e_i)+b_i^kb_{i-1}(ad~f_i)^k(f_{i-1})
-2b_i\delta_{k,2}(u_{i+1}-b_{i+1}(a_{i,i+1}+1))$.

(c) $[u_i,u_j]=0$ whenever $|i-j|>1$.

\end{proposition}

 \begin{proof} Prove (a). Indeed, 
$[u_1,u_2]=[e_1,e_2+b_1f_1]=e_{12}+b_1h_1$,
$$(ad~u_1)^2(u_2)=[e_1,e_{12}+b_1h_1]=(ad~e_1)^2(2_2)-2b_1e_1=(ad~e_1)^2(2_2)-2b_1u_1$$
$$(ad~u_2)^2(u_1)=[e_2+b_1f_1,e_{21}-b_1h_1]=(ad~e_2)^2(e_1)+b_1a_{12}e_2+b_1a_{12}e_2-2b_1^2f_1$$
$$=(ad~e_2)^2(e_1)-2b_1u_2+2b_1(a_{12}+1)e_2$$

In particular, 
$(ad~u_1)^3(u_2)=(ad~e_1)^3(e_2)$,
$$(ad~u_2)^3(u_1)=(ad~e_2)^3(e_1)+b_1[f_1,(ad~e_2)^2(e_1)]=(ad~e_2)^3(e_1) \ .$$
By taking commutators repeatedly with $u_1$ and $u_2$ respectively, we finish the proof of (a).

Prove (b).
$[u_i,u_j]=[e_i+t_kf_k,e_j+t_\ell f_\ell]=e_{ij}+t_kt_\ell f_{k\ell}+\delta_{i\ell} t_ih_i-\delta_{jk} t_jh_j$, where $e_{ij}:=[e_i,e_j]$, $f_{k\ell}=[f_k,f_\ell]$. Then 
$$(ad~u_i)^2(u_j)=[e_i+t_kf_k,e_{ij}+t_kt_\ell f_{k\ell}+\delta_{i\ell} t_ih_i-\delta_{jk} t_jh_j]$$
$$=e_{iij}+t_k^2t_\ell f_{kk\ell}+\delta_{i\ell} t_i t_k a_{ik}f_k-\delta_{i\ell}t_i(2e_i-a_{ik}t_kf_k)+\delta_{jk}a_{ji}t_je_i+\delta_{jk}t_j(a_{ji}e_i-2t_jf_j)$$
$$=e_{iij}+t_k^2t_\ell f_{kk\ell}+t_k(2\delta_{i\ell} t_i a_{ik}-2\delta_{jk}t_j)f_k+2(\delta_{jk}a_{ji}t_j-\delta_{i\ell}t_i)e_i$$
if $k\ne i$, $\ell\ne j$
where $e_{ab}:=[e_a,e_b]$, $f_{ab}:=[f_a,f_b]$, $e_{aab}:=[e_a,e_{ab}]$, $f_{aab}:=[f_a,f_{ab}]$.

Therefore, we obtain 
$(ad~u_i)^3(u_j)=(ad~e_i)^3(e_j)+t_k^3t_\ell (ad~f_k)^3(f_\ell)$
because $[e_i,f_{kk\ell}]=0$, $[f_k,e_{iij}]=0$.

By taking commutators repeatedly with $u_i$, we obtain by induction in $r$:
\begin{equation}
    \label{eq:r>2}
(ad~u_i)^r(u_j)=(ad~e_i)^r(e_j)+t_k^{r-1}t_\ell(ad~f_k)^r(f_\ell) \ .
\end{equation}
Taking $j=i+1,k=i-1,\ell=i$, $r\mapsto k$, we finish the proof of the first part of (b). Taking $i\mapsto i+1,j=k=i,\ell=i-2$, $r\mapsto k$, we finish the proof of the first part of (b) (all $t_m=b_m$).

This finishes the proof of (b).

Part (c) is immediate.

The proposition is proved.
\end{proof}

Now we can finish proof of Theorem \ref{th:conical electric2}. Indeed, in view of Proposition \ref{pr:local relations}(a) one can take $a'_{12}=a_{12}$, $a'_{21}=a_{21}$ 

Also, if $i>1$, then in view of Proposition \ref{pr:local relations}(b), one can take $a'_{i,i+1}=\min(a_{i,i+1},a_{i-1,i})$ and $a'_{i+1,i}=\min(a_{i+1,i},a_{i,i-1})$.

Finally, in view of Proposition \ref{pr:local relations}(c), one can take $a'_{ij}=a_{ij}=0$ whenever $|i-j|>1$.

Thus, Theorem \ref{th:conical electric2} is proved. \endproof

In the notation of Theorem \ref{th:conical electric3}, denote 
$$u'_i:= \begin{cases}
e_i+a_{i,i-1}a_{i-1}a_if_{i-1} & \text{if $1\le i\le r$}\\
e_i+a_{ir}a_ra_if_r & \text{if $r< i\le r+d$}\\
\end{cases}\ .$$

Clearly,  the relations for the pairs
$(u'_i,u'_{i+1})$, $i=1,\ldots,r-1$ 
 follow from Theorem \ref{th:conical electric2}.

It is also clear that $[u'_i,u'_j]=0$ whenever $i=1,\ldots,r$, $j=r+1,\ldots,r+d$ and $r<i,j\le r+d$.

Theorem \ref{th:conical electric3} is proved. \endproof

\subsection{Proof of Theorems \ref{th:conical electric2 injective}}
\label{subsec:proof of Theorem conical electric2 injective}
We retain the notation of Section \ref{subsec:Proof of Theorem conical electrical2}.
Theorem \ref{th:conical electric3} implies that that the homomorphism of 
Theorem \ref{th:conical electric2 injective} is well-defined. 
Applying the argument of Section \ref{subsec:Proof of Theorem conical electrical2} once again, we obtain the desired assertion.

Thus, Theorem \ref{th:conical electric2 injective} is proved. \endproof

\subsection{Proof of Theorem \ref{th:conical electrical2 non-conjugate}}
\label{subsec:proof of Theorem conical electrical2 non-conjugate}
Taking $\gg=\gg''=\gg$ in Theorem \ref{th:conical electric2 injective} with $r=2$ we obtain the assertion of (a).

Taking $\gg=\gg''=so_{2n}$ in Theorem \ref{th:conical electric2 injective} with the branch at the vertex $r=n-2$, $d=2$ of the Dynkin diagram $I=\{1,\ldots,n\}$ we obtain the assertion of (b).

Prove (c). Let $\gg=\widehat {sl}_n$, $I=\{1,\ldots,n\}$, $1$ and $n$ are neighbors. Taking $sl_n\subset \gg$ corresponding to $I\setminus\{n\}=\{1,\ldots,n-1\}$ gives us relations of $sl_n^{({\bf b})}$ for $u_1,\ldots,u_{n-1}$. The remaining relations of $\gg^{({\bf b})}$ can be obtained by the cyclic symmetry $i\mapsto i+1\mod n$.

Applying the argument of Section \ref{subsec:Proof of Theorem conical electrical2} once again, we see that the subalgebra generated by $u_i$, $i\in I$ is isomorphic to $\gg^{({\bf b})}$, i.e., is its edge model.

Theorem \ref{th:conical electrical2 non-conjugate} is proved. \endproof

\section{Appendix: Flat deformations}
\label{Appendix A: Flat deformations}

For any ring $R$ and any set $X$ denote by $R\langle X\rangle$ the free $R$-algebra freely generated by $X$.

Let $B$ be an algebra over $R$ and $X\subset B$ be a generating set. Denote by $\varphi_{X,B}$ the canonical surjection $R\langle X\rangle\twoheadrightarrow B$, denote $J_X:=Ker~\varphi_{X,B}$. The natural filtration  on $R\langle X\rangle$ with $\deg x=1$ for all $x\in X$ induces a natural filtration on $B$, so let $grB$ be the associated graded algebra of $B$. Clearly, $gr R\langle X\rangle=R\langle X\rangle$. 
For any subset $S\subset B$ denote by $grS$ the image of $S$ in $grB$.

If $A$ is a filtered algebra and $J$ is a two-sided ideal of $A$, then 

$gr A/J$ is canonically isomorphic to $gr A/gr J$. In particular, 
$$gr B=R\langle gr X\rangle/gr J_X\ .$$
Denote by $\varphi_{grX,grB}$ the canonical surjection $R\langle grX\rangle\twoheadrightarrow grB=R\langle gr X\rangle/gr J_X$.

\begin{proposition} 
\label{pr:flat filtred algebras}
Let $B$ be an algebra over $R$ generated by $X\subset B$. Suppose that $Y$ is a subset of $R\langle X\rangle$ such that

$\bullet$ $Y$ is contained in the kernel of $\varphi_{X,B}$.

$\bullet$ The kernel of  $\varphi_{grX,grB}$ is generated by $grY$. 

Then $B=R\langle X\rangle/\langle Y\rangle$, 
where $\langle Y\rangle$ denotes the two-sided ideal of $R\langle X\rangle$
generated by $Y$.
\end{proposition}

\begin{proof} Indeed,  for any subset $Y\subset R\langle X\rangle$, the inclusion $grY \subset gr\langle Y\rangle$ implies the inclusion of ideals $\langle grY\rangle\subset gr\langle Y\rangle$.  Now let $Y$ be as in the assumptions of the proposition. Then the inclusion $Y\subset J_X$ implies the inclusion of ideals $\langle Y\rangle\subset J_X$ hence $gr\langle Y\rangle\subset grJ_X$ (if the latter inclusion is an equality, then so is the former). Taking into account that $gr J_X=\langle gr Y\rangle$, we obtain $gr\langle Y\rangle\subset \langle grY\rangle$. 
These two opposite inclusions imply that $gr\langle Y\rangle= grJ_X$. Therefore, $J_X=\langle Y\rangle$. 

The proposition is proved.
\end{proof}

\section*{Acknowledgments}
For the third author this work is an output of a research project implemented as part of the Basic Research Program at the National Research University Higher School of Economics (HSE University).
AB and VG are grateful to the MPI in Bonn, Heidelberg University, and University of Geneva where a part of this work was accomplished.

\end{document}